\documentclass[12pt,letterpaper]{amsart}
\usepackage{amsmath,amsfonts,amssymb,amsthm,amscd}
\usepackage[english]{babel}
\usepackage[mathscr]{eucal}
\usepackage{graphicx}
\usepackage{appendix}

\renewcommand{\d}{\delta}

\newcommand{\e}{\epsilon}

\renewcommand{\l}{\lambda}

\newcommand{\vare}{\varepsilon}

\newcommand{\N}{{\mathbb N}}

\newcommand{\real}{{\mathbb R}}

\newcommand{\Z}{{\mathbb Z}}
\newcommand{\torus}{{\mathbb T}}

\newcommand{\bP}{{\mathbf P}}

\newcommand{\curl}{{\rm curl}\,}

\DeclareMathOperator{\dv}{div} %
\DeclareMathOperator{\crl}{curl} %

\newtheorem{theorem}{Theorem}
\newtheorem{proposition}{Proposition}
\newtheorem{lemma}{Lemma}

\theoremstyle{definition}
\newtheorem{definition}{Definition}

\theoremstyle{remark}

\def \e1 {\vec{e}_1}
\def \e2 {\vec{e}_2}
\def \lan {\langle}
\def \ran {\rangle}

\begin{document}


\title[Energy conservation]{On the Onsager conjecture in two dimensions}

\author[A. Cheskidov]{A. Cheskidov}
\address[A. Cheskidov]{Department of Mathematics, Statistics and Computer Science\\
University of Illinois at Chicago\\322 Science and Engineering Offices (M/C 249)\\
851 S Morgan St\\Chicago, IL 60607-7045\\USA}
\email{acheskid@math.uic.edu}
\author[M. C. Lopes Filho]{M. C. Lopes Filho}
\address[M. C. Lopes Filho]{Instituto de Matem\'atica\\Universidade Federal do Rio de Janeiro\\Cidade Universit\'aria -- Ilha do Fund\~ao\\Caixa Postal 68530\\21941-909 Rio de Janeiro, RJ -- BRAZIL.}
\email{mlopes@im.ufrj.br}\pagestyle{myheadings}
\author[H. J. Nussenzveig Lopes]{H. J. Nussenzveig Lopes}
\address[H. J. Nussenzveig Lopes]{Instituto de Matem\'atica\\Universidade Federal do Rio de Janeiro\\Cidade Universit\'aria -- Ilha do Fund\~ao\\Caixa Postal 68530\\21941-909 Rio de Janeiro, RJ -- BRAZIL.}
\email{hlopes@im.ufrj.br}
\author[R. Shvydkoy]{R. Shvydkoy}
\address[R. Shvydkoy]{Department of Mathematics, Statistics and Computer Science\\
University of Illinois at Chicago\\322 Science and Engineering Offices (M/C 249)\\
851 S Morgan St\\Chicago, IL 60607-7045\\USA}
\email{shvydkoy@math.uic.edu}

\date{\today}

\begin{abstract}
This note addresses the question of energy conservation for the 2D Euler system with an $L^p$-control on vorticity. We provide a direct argument, based on a mollification in physical space, to show that the energy of a weak solution is conserved if $\omega = \nabla \times u \in L^{\frac32}$. An example of a 2D field in the class $\omega \in L^{\frac32 - \epsilon}$ for any $\epsilon>0$, and $u\in B^{1/3}_{3,\infty}$ (Onsager critical space) is constructed with non-vanishing energy flux. This demonstrates sharpness of the kinematic argument. Finally we prove that any solution to the Euler equation produced via a vanishing viscosity limit from Navier-Stokes, with $\omega \in L^p$, for $p>1$, conserves energy. This is an Onsager-supercritical condition under which the energy is still conserved, pointing to a new mechanism of energy balance restoration.

\end{abstract}

\maketitle

\section{Introduction}

We consider the initial-value problem for the two-dimensional incompressible Euler equations on the torus $\mathbb{T}^2 \equiv [0,2\pi]^2$, with initial data $u_0 \in L^2(\mathbb{T}^2)$, which
we write:

\begin{equation} \label{eeq}
\left\{
\begin{array}{l}
\partial_t u + (u \cdot \nabla) u = -\nabla p\\
\dv u = 0\\
u(t=0) = u_0.
\end{array}
\right.
\end{equation}

We are interested in {\it weak solutions} for which the vorticity, $\omega \equiv \crl u$, is $p$-th power integrable, for some $p > 1$. More precisely, we have:

\begin{definition} \label{wksolee}
Fix $T>0$ and $u_0 \in L^2(\torus^2)$ with initial vorticity in $L^p(\torus^2)$, for
some $p > 1$. Let $u \in C_{\mathrm{weak}}(0,T;L^2(\mathbb{T}^2))$ with $\omega \in L^{\infty}(0,T;L^p(\torus^2))$. We say $u$ is a weak solution of the incompressible Euler equations with initial velocity $u_0$ if
\begin{enumerate}
\item for every test vector field $\Phi \in C^{\infty}([0,T) \times \torus^2)$ such that $\dv \Phi(t, \cdot) = 0$ the following identity holds true:
\[\int_0^T \int_{\torus^2} \partial_t \Phi \cdot u + u \cdot D\Phi u \, dxdt + \int_{\torus^2} \Phi(0,\cdot) \cdot u_0 \, dx = 0.\]

\item For almost every $t \in (0,T)$, $\mbox{ div } u(t,\cdot) = 0$, in the sense of distributions.
\end{enumerate}
\end{definition}

Existence of such weak solutions is known, see \cite{DiPM87}, but uniqueness is open, except for the case $p=\infty$.

There is little qualitative information known about weak solutions. In particular, the kinetic energy
\[E(t) = \frac12 \int_{\torus^2} |u|^2 \, dx,\]
is not known to be conserved for such solutions.  Energy conservation has been the subject of many recent publications, \cite{Shnirel,Scheff,dr,CET,DS07,ccfs,BDS14,Isett,Eyink}, in part due to its intimate connection to questions of turbulence. The minimal regularity required for $u$ to be conservative is claimed to be $1/3$ in a properly understood measurement for smoothness. This is known as the Onsager conjecture of 1949, \cite{Onsager49}.

\begin{definition} \label{conswksol}
We call $u \in L^{\infty}(0,T;L^2(\torus^2))$ a {\it conservative weak solution} if $u$ is a weak solution of the incompressible Euler equations for which $E(t)=E(0)$ for every $t \in [0,T]$.
\end{definition}
It was proved in \cite{ccfs} that, indeed, if
\begin{equation}\label{e:besov-reg}
\int_0^T \int_{\torus^n} \frac{|u(x-y,t) - u(x,t)|^3}{|y| } dx dt \to 0, \text{ as } |y| \to 0,
\end{equation}
then $u$ is conservative on the interval $[0,T]$; this condition is independent of the dimension $n>1$. Condition \eqref{e:besov-reg} represents a little better smoothness then $1/3$ in the averaged Besov space, although no rate of vanishing in \eqref{e:besov-reg} is required. In particular if $u \in L^3((0,T);B^{1/3}_{3,c_0}(\torus^n))$, where $c_0$ signifies vanishing of the Littlewood-Paley projections of $u$, see \cite{ccfs}, then \eqref{e:besov-reg} is guaranteed to hold. In the case $n=2$, on the vorticity side the above condition holds if $\omega \in L^{\infty}(0,T;L^p(\torus^2))$ for $p \geq \frac32$. Thus, $L^\frac32$ represents the Onsager-critical condition in dimension $n=2$. In the first part of this paper we will present a different and more direct argument based on a mollification of the Euler system in physical space and a control estimate on the Reynolds stress tensor.
\begin{theorem} \label{baseline}
Fix $T>0$ and let $u \in C_{\mathrm{weak}}(0,T;L^2(\torus^2))$ be a weak solution with $\omega \equiv \crl u \in L^{\infty}(0,T;L^{3/2}(\torus^2))$. Then $u$ is conservative. Moreover, the following local energy balance law holds in the sense of distributions:
\begin{equation}\label{e:loc-en}
\partial_t \left( \frac{|u|^2}{2} \right) + \dv \left[ u \left( \frac{|u|^2}{2} + p \right) \right] = 0.
\end{equation}
\end{theorem}
In particular, \eqref{e:loc-en} implies a local energy balance law as well. From the alternative Fourier viewpoint, proving the energy law requires to show that the energy flux through dyadic scales, given by $\int_{\torus^2} S_q[u] \cdot S_q[ (u \cdot \nabla) u] \,dx$ , where $S_q[u]$ is the Littlewood-Paley truncation of $u$, vanishes as $q \to \infty$. We present a similar direct argument to demonstrate this under $L^\frac32$ control on vorticity. A kinematic example of a field $u \in C^{1/3}$ for which this flux does not vanish was first presented in Eyink \cite{Eyink}. The construction is essentially three dimensional and self-similar with symmetric distribution of the Fourier spectrum. In the second part of the paper we will present an example of a 2D field $u \in B^{1/3}_{3,\infty}$, $\omega \in L^{\frac32 - \epsilon}$, for any $\epsilon >0$, with an asymmetric spectrum, that has the same property. This, firstly, demonstrates that our kinematic argument of Theorem \ref{baseline} is sharp, and secondly, that the Onsager regularity threshold does not relax even with an additional $L^p$-control on vorticity.

Although, from a purely kinematic point of view, condition \eqref{e:besov-reg} for $u$, together with an $L^\frac32$-condition for $\omega$, are sharp, to construct an actual non-conservative weak solution to the Euler system in those critical spaces proved to be a very difficult problem. First attempts were made by Scheffer and Shnirelman, \cite{Scheff,Shnirel}, with contraction of a compactly supported in space and time solution in $L^2_{t,x}$, a rather rough space. A further improvement was only possible with the influx of new ideas from topology brought by De Lellis and Szekelyhidi in \cite{DS07}. They improved regularity to $L^\infty_{t,x}$, and shortly after, in another breakthrough, to H\"{o}lder continuous classes. The state of the art results belong to Isett \cite{Isett}, with a construction leading to a non-conservative weak solution $u\in C^{1/5 - \epsilon}$,  and Buckmaster, De Lellis and Szekelyhidi \cite{BDS14}, in $L^1(\real;C^{1/3 - \epsilon}(\torus^3))$. The latter reaches the critical regularity in space but lacks proper temporal bounds. In two dimensions, Choffrut, De Lellis and Szekelyhidi constructed continuous weak solutions of the two-dimensional Euler equations with arbitrarily prescribed, smooth and positive time-dependent energy, see \cite{CDS14}. In all these constructions, however, there is no control on the integrability of vorticity.

Recently, a few results emerged that challenged the Onsager criticality by way of using additional mechanisms for the energy law to hold. As Bardos and Titi argue in \cite{BT10}, one such mechanism is the absence of nonlocal pressure, and hence pure transport nature of the system. They observe that the famous DiPerna-Majda example of an evolving parallel shear flow
\[
u = \lan u_1(x_2),0, u_3(x_1 - t u(x_2) \ran
\]
conserves energy even if only $u_1,u_3 \in L^2$. Another mechanism isolated in \cite{shv-org} relies on geometric assumptions on the singularity set for Onsager-critical solutions, such as vortex sheets. The Euler system enforces a kinematic condition on the interface -- particles remain on the sheet at all times -- thus there is no exchange of energy between the two sides of the sheet. In the case of critical homogeneous solutions, as observed in \cite{LS15}, the mechanism for energy conservation is the Hamiltonian structure of the system that induces extra symmetry on the flow streamlines near the singularity at the origin. In the third part of this paper we will reveal another, even more dramatic Onsager-supercritical condition under which the energy is still conserved -- vanishing viscosity limit under $L^p$-control on vorticity for \emph{any} $p>1$.

\begin{definition} \label{physsol}
Let $u \in C(0,T;L^2(\torus^2))$. We say that $u$ is a {\em physically realizable weak solution of the incompressible 2D Euler equations} with initial velocity $u_0 \in L^2(\torus^2)$ if the following conditions hold.
\begin{enumerate}
\item $u$ is a weak solution of the Euler equations in the sense of Definition \ref{wksolee};
\item there exists a family of solutions of the incompressible 2D Navier-Stokes equations with viscosity $\nu >0$, $\{u^{\nu}\}$, such that, as $\nu \to 0$,
    \begin{enumerate}
    \item $u^{\nu} \rightharpoonup u$ weakly$^*$ in $L^{\infty}(0,T;L^2(\torus^2))$;
    \item $u^{\nu}(0,\cdot)\equiv u_0^{\nu} \to u_0$ strongly in $L^2(\torus^2)$.
    \end{enumerate}
\end{enumerate}
\end{definition}

\begin{theorem} \label{physwksoltns}
Let $u \in C(0,T;L^2(\torus^2))$ be a physically realizable weak solution of the incompressible 2D Euler equations. Suppose that $u_0 \in L^2$ is such that $\curl u_0 \equiv \omega_0 \in L^p(\torus^2)$, for some $p>1$. Then $u$ conserves energy.
\end{theorem}

The proof of Theorem \ref{physwksoltns} essentially claims that the total energy dissipation rate $\epsilon_\nu = \nu \int_0^T \| \nabla u^\nu \|_{L^2}^2 dt$ vanishes as $\nu \to 0$. This claim demonstrates a striking difference between the 3D and 2D cases, where in the 3D case the non-vanishing limit $\epsilon_\nu \to \epsilon_0>0$ is the basic postulate of the classical Kolmogorov theory of fully developed turbulence.

\subsection*{Acknowledgments}
The authors are grateful for the hospitality of IMPA, in Rio de Janeiro, where part of this work was done. In addition, H.J.N.L. thanks the University of Illinois at Chicago for its hospitality. The authors would like to thank Peter Constantin for an interesting comment. H.J.N.L.'s research was supported in part by CNPq grant \# 307918/2014-9 and FAPERJ project \# E-26/103.197/2012. M.C.L.F.'s work was partially funded by CNPq grant \# 306886/2014-6. The work of R.S. was partially supported by NSF grants DMS-1210896 and DMS-1515705. A.C. was partially supported by NSF grants
DMS-1108864 and DMS-1517583.

\section{Energy Conservation under $L^p$-control of vorticity}

We begin by presenting the proof of Theorem \ref{baseline}.

\begin{proof}

Let $\vare \in (0,1/2)$ and choose a Friedrichs mollifier $\xi \in C^{\infty}(\real^2)$, $0\leq \xi \leq 1$, even, $\int \xi (x) \, dx = 1$. Set
\[\xi_{\vare}=\xi_{\vare}(x)=\frac{1}{\vare^2} \xi\left( \frac{x}{\vare} \right),\]
and define $\zeta_{\vare}(x) = \xi_{\vare}(x)$, for $x \in [-1/2,1/2) \times [-1/2,1/2)$,
extended periodically to a family of mollifiers in $C^{\infty}(\torus^2)$.

We take the convolution of the Euler equations \eqref{eeq} with $\zeta_{\vare}$ and write $u^{\vare} = \zeta_{\vare} \ast u$, $p^{\vare}=\zeta_{\vare} \ast p$. We find:
\begin{equation} \label{mollifeq}
\partial_t u^{\vare} + (u^{\vare}\cdot \nabla) u^{\vare} = - \nabla p^{\vare} + \mathcal{R}^{\vare},
\end{equation}
 with
\[
 \mathcal{R}^{\vare} \equiv (u^{\vare}\cdot \nabla) u^{\vare} - \zeta_{\vare} \ast [(u \cdot \nabla )u].
 \]
Indeed, this can be rigorously justified by using, as a test vector field, $\Phi^{\vare} = \Phi \ast \zeta_{\vare}$ in Definition \ref{wksolee}, with equation \eqref{mollifeq} understood in the
sense of distributions. Moreover, since $u^{\vare}$, $p^{\vare}$ and $\mathcal{R}^{\vare}$ are
smooth in space, it follows from \eqref{mollifeq} that $\partial_t u^{\vare}$ is smooth in space,
a.e. in time.
Since $u^{\vare}$ is smooth in space,  from equation \eqref{mollifeq} we find that $u^{\vare}$ is Lipschitz continuous in time. Hence, we may multiply the equation by $u^{\vare}$ to obtain:
\begin{equation} \label{molifenbaleq}
\partial_t\left(\frac{|u^{\vare}|^2}{2}\right) + \dv \left[ u^{\vare} \left( \frac{|u^{\vare}|^2}{2} + p^{\vare} \right) \right] = u^{\vare} \cdot  \mathcal{R}^{\vare}.
\end{equation}
We claim that, as $\vare \to 0$, we have:
\begin{enumerate}
\item[(A)] $\partial_t\left(\frac{|u^{\vare}|^2}{2}\right) \to \partial_t\left(\frac{|u|^2}{2}\right)$ in the sense of distributions;
\item[(B)] $\dv \left[ u^{\vare} \left( \frac{|u^{\vare}|^2}{2} + p^{\vare} \right) \right] \to
\dv \left[ u \left( \frac{|u|^2}{2} + p \right) \right]$ in the sense of distributions;
\item[(C)] $u^{\vare} \cdot  \mathcal{R}^{\vare} \to 0$ strongly in $L^{\infty}(0,T;L^1(\torus^2))$.
\end{enumerate}
We first note that, as $\omega \in  L^{\infty}(0,T;L^{3/2})$, it follows that $u \in  L^{\infty}(0,T;W^{1,3/2})$. Hence, by the Sobolev imbedding theorem, $u \in  L^{\infty}(0,T;L^6)$.

To establish item (A) it is enough to show that \[\frac{|u^{\vare}|^2}{2} \to \frac{|u|^2}{2}\]
in $L^{\infty}(0,T;L^1(\torus^2))$. We show more: the convergence holds true in the strong topology of $L^{\infty}(0,T;L^{6/5}(\torus^2))$.
Indeed,
\[\left\|\frac{|u^{\vare}|^2}{2} - \frac{|u|^2}{2}\right\|_{L^{\infty}(L^{6/5})} = \left\|\frac{1}{2}(u^{\vare} - u)(u^{\vare} + u)\right\|_{L^{\infty}(L^{6/5})}\leq
\frac{1}{2}\|u^{\vare} - u\|_{L^{\infty}(L^6)} \|u^{\vare} + u\|_{L^{\infty}(L^{3/2})} \to 0,\]
because $u^{\vare} = \zeta^{\vare} \ast u$, $u^{\vare}$ is uniformly continuous in $(0,T)$ and since convolutions are strongly continuous in Lebesgue spaces. Here we use the generalized H\"{o}lder inequality
$\|f \cdot g\|_{L^{6/5}} \leq \|f\|_{L^6} \|g\|_{L^{3/2}}$.

We now prove item (B). It is enough to show that
\[u^{\vare} \left( \frac{|u^{\vare}|^2}{2} + p^{\vare} \right)  \to
 u \left( \frac{|u|^2}{2} + p \right)  \]
 in $L^{\infty}(0,T;L^1(\torus^2))$. To show this we write:
 \begin{equation} \label{newwrite}
 \begin{array}{l}
 u^{\vare} \left( \frac{|u^{\vare}|^2}{2} + p^{\vare} \right)  -
 u \left( \frac{|u|^2}{2} + p \right)= \\ \\
 (u^{\vare} - u)\frac{|u^{\vare}|^2}{2} + u \left( \frac{|u^{\vare}|^2}{2} - \frac{|u|^2}{2}\right) + (u^{\vare} - u) p^{\vare} + u (p^{\vare} -p).
 \end{array}
 \end{equation}
Furthermore, by taking the divergence of the Euler equations \eqref{eeq} we find
\[-\Delta p = \mbox{ div}\mbox{ div} (u \otimes u), \]
and $u \otimes u \in L^{\infty}(0,T;L^3(\torus^2))$. Therefore, taking advantage of the fact that the flow domain is the torus, we deduce that
$p \in L^{\infty}(0,T;L^3(\torus^2))$.
We estimate the terms in \eqref{newwrite} as follows, recalling that $p^{\vare} = \zeta^{\vare} \ast p$,
\begin{equation*}
\begin{array}{l}
\left\| u^{\vare} \left( \frac{|u^{\vare}|^2}{2} + p^{\vare} \right)  -
 u \left( \frac{|u|^2}{2} + p \right) \right\|_{L^{\infty}(L^1)} \leq\\ \\
\|u^{\vare} - u\|_{L^{\infty}(L^3)} \left\| \frac{|u^{\vare}|^2}{2} \right\|_{L^{\infty}(L^{3/2})} +  \|u\|_{L^{\infty}(L^6)} \left\|\frac{|u^{\vare}|^2}{2} - \frac{|u|^2}{2}\right\|_{L^{\infty}(L^{6/5})} \\
+  \|u^{\vare} - u\|_{L^{\infty}(L^{3/2})}\| p \|_{L^{\infty}(L^3)}+\|u\|_{L^{\infty}(L^{3/2})} \|p^{\vare} -p\|_{L^{\infty}(L^3)} \to 0,
\end{array}
\end{equation*}
as before.

Finally, we analyze item (C), the flux term. We will show that $\mathcal{R}^{\vare} \to 0$ strongly in $L^{\infty}(0,T;L^{6/5}(\torus^2))$. This is enough to prove item (C). Since $u^{\vare} $ is bounded in $L^{\infty}(0,T;L^6(\torus^2))$. We have:
\begin{equation*}
\begin{array}{l}
\|\mathcal{R}^{\vare}\|_{L^{\infty}(L^{6/5})} = \|(u^{\vare}\cdot \nabla) u^{\vare} - \zeta_{\vare} \ast [(u \cdot \nabla )u]\|_{L^{\infty}(L^{6/5})} \\ \\
\leq \|(u^{\vare}\cdot \nabla) (u^{\vare}-u)\|_{L^{\infty}(L^{6/5})}  + \|(u^{\vare} - u)\cdot \nabla u\|_{L^{\infty}(L^{6/5})} +
\|(u \cdot \nabla) u -\zeta_{\vare} \ast [(u \cdot \nabla )u]\|_{L^{\infty}(L^{6/5})} \\ \\
\leq \|u^{\vare}\|_{L^{\infty}(L^6)}\|\nabla u^{\vare} - \nabla u\|_{L^{\infty}(L^{3/2})} + \|u^{\vare} - u\|_{L^{\infty}(L^6)}\|\nabla u\|_{L^{\infty}(L^{3/2})}  \\ \\
+ \| (u \cdot \nabla) u- \zeta_{\vare} \ast [(u \cdot \nabla )u]\|_{L^{\infty}(L^{6/5})} \to 0,
\end{array}
\end{equation*}
because $u^{\vare} \to u$ in $L^{\infty}(L^6(\torus^2))$, $\nabla u^{\vare} = \zeta^{\vare} \ast \nabla u \to \nabla u$ in $L^{\infty}(L^{3/2}(\torus^2))$ and $u\cdot \nabla u \in L^{\infty}(L^{6/5}(\torus^2))$. We are using repeatedly that convolutions are continuous, in the strong topology, on Lebesgue spaces.

This concludes the proof.

\end{proof}

The criticality of the exponent $p = 3/2$ in the proof above was required only for the weak continuity
of the Reynolds stress term $u \mathcal{R}$ in the energy balance. In fact, it can be verified that weak continuity of the other terms would still work provided that $p > 6/5$.

A natural question at this point is whether there exist weak solutions of the two-dimensional Euler equations with vorticity in $L^p$, $p<3/2$,  which are not conservative. Construction of such examples, perhaps using convex integration as in \cite{DS07,CDS14,BDS14}, is an open problem. Our next step is to present an example that shows that the exponent $p=3/2$ in the proof of Theorem \ref{baseline} is sharp. First we recall basic terminology and notation of Littlewood-Paley theory on the torus $\torus^2$.

For $f \in C^{\infty}(\torus^2)$ and $\alpha \in \mathbb{Z}^2$, we denote the Fourier coefficients for $f$ by $\widehat{f}(\alpha)$. More precisely,
\[\widehat{f}(\alpha) = \int_{\torus^2} e^{-2 \pi i \alpha \cdot x}f(x) dx,
\mbox{ and }
f(x) = \sum_{\alpha \in \mathbb{Z}^2} \widehat{f}(\alpha) e^{2 \pi i \alpha \cdot x}.\]
Set $\lambda_q = 2^q$. We fix a nonnegative radial function $\chi = \chi(\xi)$ in $C^{\infty}_c(\real^2)$, with support contained in the unit ball $B(0;1)$, and such that $\chi(\xi) \equiv 1$ if $|\xi| \leq 1/2$.  We define
\[\varphi(\xi) = \chi(\lambda_1^{-1} \xi) - \chi(\xi).\]
With this notation, the $q$-th Littlewood-Paley component of $f$ is given by:
\[\Delta_q f = \sum_{\alpha \in \mathbb{Z}^2} \varphi(\lambda_q^{-1} \alpha) \widehat{f}(\alpha) e^{2 \pi i \alpha \cdot x}. \]
We introduce notation for the $q$-th Littlewood-Paley truncation :
\[S_q[f] = \widehat{f}_{(0,0)} + \sum_{p \leq q-1} \Delta_p f = \sum_{\alpha \in \mathbb{Z}^2}
\chi(\lambda_q^{-1} \alpha) \widehat{f}(\alpha) e^{2 \pi i \alpha \cdot x}.\]
By testing the Euler system with $S_q[S_q[u]]$ we readily see that the proof of the global energy conservation reduces to showing that the energy flux
\begin{equation}\label{}
\Pi_q[u] = \int_{\torus^2} S_q[u] \cdot S_q[ (u \cdot \nabla) u] \,dx
\end{equation}
vanishes on average in time as $q\to \infty$. This holds, in fact, pointwise in time for any field with $L^\frac32$ control on $\omega$.

\begin{proposition} \label{motiv}
Let $u$ be a divergence-free vector field in $W^{1,3/2}(\torus^2)$.  Then $\Pi_q[u] \to 0$
as $q \to \infty$.
\end{proposition}

\begin{proof}

Let $H_q = H_q(x) = \sum_{\alpha\in \Z^2} \chi(\lambda_q^{-1}\alpha) e^{2\pi i \alpha \cdot x}$.
It can be easily verified that
\[S_q[f] = \int_{\torus^2} H_q(x-y)f(y)\,dy.\]
Hence, as $u$ is divergence-free it follows that $\dv (S_q[u])=0$. Similarly, we have $\nabla S_q[u]=S_q[\nabla u]$. Clearly, $S_q[u]$ is a smooth function. In light of these observations it holds that
\begin{equation} \label{trilinear}
\int_{\torus^2} S_q[u]\cdot \{(S_q[u]\cdot\nabla)S_q[u]\} \,dx=0.
\end{equation}
Hence,
\[\int_{\torus^2} S_q[u] \cdot S_q[ (u \cdot \nabla) u] \,dx = \int_{\torus^2}S_q[u]\cdot \{(S_q[u]\cdot\nabla)S_q[u] - S_q[ (u \cdot \nabla) u] \} \,dx.\]

As $u \in W^{1,3/2}(\torus^2) \subset L^6(\torus^2)$, in order to establish the proposition it is enough to show that
\[
\|(S_q[u]\cdot\nabla)S_q[u] - S_q[ (u \cdot \nabla) u] \|_{L^{6/5}(\torus^2)} \to 0,
\]
as $q \to \infty$. This follows analogously to the proof of item (C) in the proof of Theorem \ref{baseline}. Recall that  $S_q[u] \to u$ in $L^p$ for any $1<p<\infty$. We have:

\begin{equation*}
\begin{array}{l}
\|(S_q[u]\cdot\nabla)S_q[u] - S_q[ (u \cdot \nabla) u] \|_{L^{6/5}} \\ \\
\leq
\|(S_q[u]\cdot\nabla)(S_q[u] -u)  \|_{L^{6/5}} +
\|(S_q[u]-u)\cdot\nabla)u \|_{L^{6/5}}+
\|(u\cdot\nabla)u -S_q[ (u \cdot \nabla) u]\|_{L^{6/5}} \\ \\
\leq \|S_q[u]\|_{L^6}\|\nabla S_q[u] -\nabla u\|_{L^{3/2}} + \|S_q[u]-u\|_{L^6}\|\nabla u\|_{L^{3/2}} +
\|(u\cdot\nabla)u -S_q[ (u \cdot \nabla) u]\|_{L^{6/5}} \\ \\
\to 0,
\end{array}
\end{equation*}
as $q \to \infty$.

\end{proof}

\begin{theorem} \label{counter}
There exists a divergence free vector field $u \in B^{1/3}_{3,\infty} \cap W^{1,p}(\torus^2)$, for any $1 \leq p < 3/2$, such that $\limsup_{q\to \infty} \Pi_q[u]  \neq 0$.
\end{theorem}

\begin{proof}

Note that, if $u$ is divergence-free, we can rewrite the formula for the flux as
\[\Pi_q [u] = \int_{\torus^2} S_q[u\otimes u] \cdot \nabla S_q[u] \, dx.\]
We will begin the construction of the desired vector field $u$ by first producing a skeleton vector field $\mathcal{U} $ which belongs to $B^{1/3}_{3,\infty}(\torus^2)$ and satisfies $\limsup_{q \to \infty} \Pi_q[\mathcal{U}]  \neq 0$. We will subsequently modify it in order to place it in all Sobolev spaces $W^{1,p}$ for $p<\frac{3}{2}$.
A word about notation: first and second components of points and of vector fields are indicated by brackets $\lan \cdot, \cdot \ran$, and $\delta_{\lan a,b \ran} $ denotes the Kronecker delta supported at $\lan a, b \ran \in \mathbb{Z}^2$.

We start the construction of the skeleton field by introducing ${u}_q$ the basic local interaction triple with Fourier modes placed into the contiguous pair of $q$-th and $(q-1)$-th shells:
\[
\begin{split}
 {u}_q & = v_{q-1} + w_q ,\\
\widehat{v_{q-1}} & =   i\lan \l_{q-1}^{-1/3},0 \ran (\delta_{\lan 0,\l_{q-1} \ran}- \delta_{\lan 0, - \l_{q-1} \ran}),\\
\widehat{w_q }& = \lan 0, \l_q^{-1/3} \ran (\delta_{\lan \l_q ,0 \ran} + \delta_{\lan - \l_q ,0 \ran} ) + \l_{q}^{-1/3} \lan -1,2 \ran ( \d_{\lan \l_q,\l_{q-1} \ran} +  \d_{\lan -\l_q,-\l_{q-1} \ran}).
\end{split}
\]
The skeleton field will be defined as a superposition of these interaction triples, $\mathcal{U} = \sum_{j}  {u}_{q_j}$, where $q_j$ will be chosen as a well-separated unbounded sequence. Since we have $\l_q^{1/3} \|  {u}_q \|_3 \sim 1$, it follows that  $\mathcal{U} \in B^{1/3}_{3,\infty}$.

We need to calculate $\Pi_q[ {u}_q]$. Clearly, $ {u}_q$ is a divergence free vector field, hence:
\[\Pi_q[ {u}_q] = \int_{\torus^2} S_q[ {u}_q \otimes  {u}_q] : \nabla S_q[ {u}_q]\, dx,\]
where $M : N = \sum_{i,j=1}^2 M_{ij} N_{ij}$, for any $2\times 2$ matrices $M$, $N$.
Using the Plancherel theorem it follows that
\[\Pi_q[ {u}_q]  = \sum_{\alpha \in \mathbb{Z}^2} \widehat{ S_q[ {u}_q \otimes  {u}_q]} :
\overline{\widehat{\nabla S_q[ {u}_q]}}. \]

We claim that
\[\widehat{\nabla S_q[ {u}_q]} =
\left[
\begin{array}{lrc}
0 & & -2\pi\l_{q-1}^{2/3} (\delta_{\lan 0, \l_{q-1}\ran} + \delta_{\lan 0, -\l_{q-1}\ran}) \\
& & \\
0 & & 0
\end{array}
\right].
\]
Indeed,
\[\widehat{\nabla S_q[ {u}_q]} (\alpha) =
2\pi i \left[
\begin{array}{lcr}
 \alpha_1 \widehat{S_q[ {u}^1_q]}  (\alpha) &  & \alpha_2 \widehat{S_q[ {u}^1_q]}  (\alpha)\\
& & \\
\alpha_1 \widehat{S_q[ {u}^2_q]} (\alpha) &  & \alpha_2 \widehat{S_q[ {u}^2_q]} (\alpha)\\
\end{array}
\right].\]
Now, $\widehat{ S_q[ {u}_q]} (\alpha) =
\chi(\l_q^{-1}\alpha) \widehat{  {u}_q} (\alpha)$,
 so that only the Fourier coefficients of $ {u}_q$ supported in $\{\alpha \in \mathbb{Z}^2 \, | \, |\alpha| < \l_q \}$ survive.
Therefore, $\widehat{S_q[ {u}^1_q]} =   i\l_{q-1}^{-1/3} (\delta_{\lan 0,\l_{q-1} \ran}- \delta_{\lan 0, - \l_{q-1} \ran})$
and  $\widehat{S_q[ {u}^2_q]} = 0$.
The desired claim follows easily.

In view of the calculation above it is enough to identify $\widehat{ S_q[ {u}^1_q   {u}^2_q]} $. Let us begin with the simple observation that, if $f$ and $g$ are smooth functions on $\torus^2$ such that
\[\widehat{f}(\alpha) = K_1 \delta_{\lan a , b \ran}(\alpha) \text{ and } \widehat{g}(\alpha) = K_2 \delta_{\lan c , d \ran}(\alpha), \text{ } \alpha \in \mathbb{Z}^2,\]
 for some constants $K_1$ and $K_2$, then
\[\widehat{fg}(\alpha) = K_1 K_2\delta_{\lan a+c , b+d \ran}(\alpha).\]

We use this to compute $\widehat{ S_q[ {u}^1_q   {u}^2_q]}$:
\[
\widehat{ S_q[ {u}^1_q   {u}^2_q]} = \chi(\l_q^{-1}\alpha) \widehat{ {u}^1_q   {u}^2_q}.
\]
Again, only the Fourier coefficients of $ {u}^1_q   {u}^2_q$ supported in $\{\alpha \in \mathbb{Z}^2 \, | \, |\alpha| < \l_q \}$ are relevant. We find:
\[
\widehat{ S_q[ {u}^1_q   {u}^2_q]} =  -\l_q^{-2/3} \delta_{\lan 0, \l_{q-1} \ran}
-\l_q^{-2/3} \delta_{\lan 0, -\l_{q-1} \ran} -4\l_q^{-2/3} \delta_{\lan 0, 0 \ran}.
\]

We obtain, from these calculations, that
\[\Pi_q[ {u}_q]  = \sum_{\alpha \in \mathbb{Z}^2} \widehat{ S_q[ {u}^1_q   {u}^2_q]} \cdot
\overline{2\pi i \alpha_2 \widehat{S_q[ {u}^1_q]}} \]
\[= \sum_{\alpha \in \mathbb{Z}^2} [-\l_q^{-2/3} \delta_{\lan 0, \l_{q-1} \ran}
-\l_q^{-2/3} \delta_{\lan 0, -\l_{q-1} \ran} -4\l_q^{-2/3} \delta_{\lan 0, 0 \ran}
] \cdot
\overline{-2\pi \l_{q-1}^{2/3} (\delta_{\lan 0, \l_{q-1}\ran} + \delta_{\lan 0, -\l_{q-1}\ran})} \]
\[=  4 \pi.\]

We now introduce the skeleton field $\mathcal{U} = \sum_j  {u}_{q_j}$, for some sequence $q_j \to \infty$, chosen such that
$q_j \ll q_{j+1}$. Let us split $\mathcal{U}$ into three pieces:
\begin{eqnarray*}
\mathcal{U} & =  \displaystyle{\sum_{\ell \leq j-1}}  {u}_{q_{\ell}} & +  {u}_{q_j}   + \sum_{\ell \geq j+1}  {u}_{q_{\ell}} \\
  & \equiv  \zeta_L & +   {u}_{q_j} + \zeta_H.
\end{eqnarray*}

We calculate $\Pi_{q_j}[\mathcal{U}]$:
\begin{eqnarray*}
\Pi_{q_j}[\mathcal{U}] = \Pi_{q_j}[\zeta_L  +   {u}_{q_j} + \zeta_H]  \\
= \sum_{\alpha \in \mathbb{Z}^2} \widehat{ S_{q_j}[\mathcal{U} \otimes \mathcal{U}]} :
\overline{\widehat{\nabla S_{q_j}[\zeta_L +  {u}_{q_j} + \zeta_H]}}.
\end{eqnarray*}
In principle there are twenty-seven terms to be analyzed. We begin by noting that
\[\text{supp } ( \widehat{{u_{q_{\ell}}}}) \subset \{\alpha \in \mathbb{Z}^2 \;|\; \l_{q_{\ell} - 1} \leq |\alpha| \leq \l_{q_{\ell} + 1}\}.\]
Therefore, it follows that
\[\text{supp } (\widehat{\zeta_L}) \subset  \{\alpha \in \mathbb{Z}^2 \;|\;   |\alpha| \leq \l_{q_{j-1} + 1}\} \text{ and }
\text{supp } (\widehat{\zeta_H}) \subset  \{\alpha \in \mathbb{Z}^2 \;|\;   |\alpha| \geq \l_{q_j}\}.
\]
Choose the sequence $q_j$ so that $q_{j-1} + 1 < q_j - 1$, for all $j\in \mathbb{N}$. Then, as  $\text{supp } (\widehat{\zeta_L}) \subset  \{\alpha \in \mathbb{Z}^2 \;|\;   |\alpha| \leq \l_{q_j - 1}\}$, we find
\[S_{q_j}[\zeta_L] = \zeta_L.\]
In addition, clearly we have
\[S_{q_j}[\zeta_H]=0.\]
Hence we can rewrite $\Pi_{q_j}[\mathcal{U}]$ as:
\[
\Pi_{q_j}[\mathcal{U}] = \Pi_{q_j}[\zeta_L  +   {u}_{q_j} + \zeta_H]  \]
\[= \sum_{\alpha \in \mathbb{Z}^2} \widehat{ S_{q_j}[\mathcal{U} \otimes \mathcal{U}]} :
\overline{\widehat{\nabla S_{q_j}[\zeta_L +  {u}_{q_j} ]}} \]
\[= \sum_{\alpha \in \mathbb{Z}^2} \widehat{ S_{q_j}[\mathcal{U} \otimes \mathcal{U}]} :
\overline{\widehat{\nabla \zeta_L }} \]
\[+
\sum_{\alpha \in \mathbb{Z}^2} \widehat{ S_{q_j}[\mathcal{U} \otimes \mathcal{U}]} :
\overline{\widehat{\nabla S_{q_j}[  {u}_{q_j} ]}}\]
\[\equiv \Pi^L_{q_j}[\mathcal{U}] + \Pi^{inter}_{q_j}[\mathcal{U}].
\]
Observe that, since $\zeta_L$ has no Fourier coefficients larger than $\l_{q_j-1}$, we can remove $S_{q_j}$ in the expression for $\Pi^L_{q_j}[\mathcal{U}]$ and we find
\[\Pi^L_{q_j}[\mathcal{U}]= \sum_{\alpha \in \mathbb{Z}^2} \widehat{ \mathcal{U} \otimes \mathcal{U} } :
\overline{\widehat{\nabla \zeta_L }}.\]
Here we have nine terms to deal with.

We note that
\[\sum_{\alpha \in \mathbb{Z}^2} \widehat{  \zeta_L  \otimes \mathcal{U} } :
\overline{\widehat{\nabla \zeta_L }} = \int_{\torus^2} [\mathcal{U} \cdot \nabla \zeta_L ]\cdot \zeta_L \, dx = 0,\]
since $\mathcal{U}$ is a divergence-free vector field.
We are left with:
\[ \Pi^L_{q_j}[\mathcal{U}] = \sum_{\alpha \in \mathbb{Z}^2} \widehat{  ( {u_{q_j}} + \zeta_H)  \otimes \mathcal{U} } :
\overline{\widehat{\nabla \zeta_L }}.\]

All the terms vanish, as the support of the Fourier coefficients of $ ( {u_{q_j}} + \zeta_H)  \otimes \mathcal{U} $ is  disjoint from the support of the Fourier coefficients of $\zeta_L $.

It remains to analyze $\Pi^{inter}_{q_j}[\mathcal{U}]$. We write:
\[\Pi^{inter}_{q_j}[\mathcal{U}] =
\sum_{\alpha \in \mathbb{Z}^2} \widehat{ S_{q_j}[\mathcal{U} \otimes \mathcal{U}]} :
\overline{\widehat{\nabla S_{q_j}[  {u}_{q_j} ]}}\]
\[=\Pi_{q_j}[ {u_{q_j}}] \]
 \[+ \int_{\torus^2} S_{q_j}[\zeta_L \otimes \zeta_L + \zeta_L\otimes  {u_{q_j}} +  \zeta_L \otimes \zeta_H +  {u_{q_j}}\otimes\zeta_L +  {u_{q_j}}\otimes\zeta_H + \zeta_H \otimes \zeta_L + \zeta_H \otimes  {u_{q_j}} + \zeta_H \otimes \zeta_H ] : \nabla S_{q_j}[ {u_{q_j}}]\,dx.\]

Again, the support of the Fourier coefficients of $S_{q_j}[\zeta_L \otimes \zeta_L + \zeta_L\otimes  {u_{q_j}} +  \zeta_L \otimes \zeta_H +  {u_{q_j}}\otimes\zeta_L +  {u_{q_j}}\otimes\zeta_H + \zeta_H \otimes \zeta_L + \zeta_H \otimes  {u_{q_j}} + \zeta_H \otimes \zeta_H]$ is disjoint from the support of the Fourier coefficients of $S_{q_j}[ {u_{q_j}}]$, so that the integral term vanishes. We deduce that
\[\Pi_{q_j}[\mathcal{U}]=\Pi_{q_j}[ {u_{q_j}}] = 4\pi.\]

This establishes that
\[\limsup_{q\to\infty} \Pi_{q}[\mathcal{U}] \neq 0.\]

As of now the constructed field belongs to  $ B^{1/3}_{p,\infty}$, for any $p>1$. This is because in each dyadic shell the field has only two or four non-zero modes, hence it is uniformly distributed over the domain, and hence all the $L^p$-norms there are comparable. In order to gain smoothness by reducing the integrability we have to saturate Bernstein's differential inequalities so that
\[
\| u_q\|_3 \sim \l_q^{ \frac{2}{p}- \frac{2}{3}} \|u_q\|_p,
\]
for all $p<3$, and at the same time preserve the flux calculus above. To achieve this we use lattice blocks of frequencies
around the basic wave-vectors in our skeleton field.
Namely, fix a small $\epsilon > 0$. Let $B_q = [-\l_q,\l_q]^2 \cap \Z^2$, and define
\[
\begin{split}
{u}_q & = \bP( v_{q-1} + w_q) \\
\widehat{{v}_{q-1}} & =   i\lan \l_{q-1}^{-5/3},0 \ran (\chi_{\lan 0,\l_{q-1} \ran + \epsilon B_{q-1} }- \chi_{\lan 0, - \l_{q-1} \ran+ \epsilon B_{q-1}})\\
\widehat{{w}_q} & = \lan 0, \l_q^{-5/3} \ran (\chi_{\lan \l_q ,0 \ran+ \epsilon B_q} + \chi_{\lan - \l_q ,0 \ran+ \epsilon B_q} ) + \l_{q}^{-5/3} \lan -1,2 \ran ( \chi_{\lan \l_q,\l_{q-1} \ran + \epsilon B_q} +  \chi_{\lan -\l_q,-\l_{q-1} \ran + \epsilon B_q})\\
u &= \sum_j u_{q_j},
\end{split}
\]
where $\bP$ is the Leray-Hopf projection, and $q_j$ is the sequence as before. Notice that if $\epsilon$ is small, $\bP$ preserves the $L^p$-norms of each dyadic block up to an absolute constant.
To actually compute the $L^p$ norm of ${u}_q$ let us notice that each lattice block on the physical side has a similar form. It is a modulated and rescaled Dirichlet kernel in each variable: $\l_q^{-5/3}D_{\epsilon \l_{q}}^{\otimes 2}$. By the well-known formula, $\|D_n\|_p \sim n^{1-\frac1p}$, $p>1$. Thus, $\|{u}_q\|_p \sim \l_q^{-5/3} ( \epsilon \l_q)^{2 - 2/p}) \sim \l_q^{\frac13 - \frac2p}$. So, in particular, $\|{u}_q\|_3 \sim \l_q^{-1/3}$, thus $u\in  B^{1/3}_{3,\infty}$. At the same time, for $p<\frac32$ we have $\|{u}_q\| \sim \l_q^{-1-\d}$ for some $\d>0$. So, $u \in W^{1,p}$.

When $\epsilon$ is small but fixed, the blocks on the Fourier side have a set of interactions similar to the skeleton ones, except each wavevector $\xi$ in the $y$-axis block of $v_{q-1}$ will receive an order of $\epsilon \l_q$ pairs of $\eta_1, \eta_2$, from the respective blocks in $w_{q}$ so that $\eta_1+\eta_2 = \xi$. Each such interaction gives a term of order $c \l_q^{-15/3}\l_q = c \l_q^{-4}$. Multiplied by the number of pairs $\eta_1, \eta_2$, $\epsilon \l_q^2$, and the number of $\xi$'s, $\epsilon \l_{q-1}^2$, this gives ${\Pi}_q[u] \sim c \epsilon^2 $, again independent of $q$. We can further increase the size of the gaps $q_{j+1} - q_j$ if necessary to account for small fixed $\epsilon^2$ factor and to make sure the residual terms that result from interactions between not the main triple are yet smaller. This finishes the construction.

\end{proof}

\section{Vanishing viscosity limit}

We have, at this point, a condition for a general weak solution to be conservative and an example to the effect that purely kinematic considerations will not lead to supercrititial conservative weak solutions. We emphasize  that uniqueness of weak solutions (in the sense of Definition \ref{wksolee}) has not been established unless $p=\infty$. Therefore, it is reasonable to seek supercritical conservative weak solutions depending on how the solution was obtained.
There are several available proofs of existence of weak solutions; they all use an  approximation scheme. For weak solutions with $L^p$ vorticity, the solution is obtained as a strong limit, in the topology of  $L^{\infty}(0,T;L^2(\torus^2))$, of a sequence $\{u^n\}_{n \in \N}$ of approximations. Examples of these approximation schemes include smoothing out the initial data and exactly solving the Euler equations, solving the Navier-Stokes equations and passing to the vanishing viscosity limit, using the vortex blob approximation, truncating the initial data, using the central difference scheme, {\it etc}. We say the approximation scheme is {\it conservative} if, for each $n$, $E_n(t) \equiv E_{u^n}(t) =  E_n(0)$ for all $t \in [0,T]$.

    We are interested in conditions under which weak solutions are conservative. We begin with the observation that, for flows with $L^p$ vorticity, $p>1$, if a weak solution $u$ is obtained as a limit of a conservative approximation scheme, then $u$ itself will be conservative. This follows from the analysis carried out in \cite{DiPM87}, see also \cite{LNT00}. Clearly, this is the case for solutions obtained by smoothing initial data and exactly solving the 2D Euler equations with the smooth data, as was observed in \cite[Section 5]{LNT00}.

We will now turn to weak solutions of the 2D Euler equations which are obtained as weak limits of solutions of the Navier-Stokes equations, as viscosity vanishes. We will see that, for these solutions, the critical regularity for energy conservation is also below the critical Onsager exponent $3/2$.

\begin{lemma} \label{noconcentrateviscous}
Let $u \in C(0,T;L^2(\torus^2))$ be a physically realizable weak solution of the incompressible 2D Euler equations and  let $\{u^{\nu}\}$ be a family of solutions to the Navier-Stokes equations with viscosity $\nu$ satisfying (2a) and (2b) of Definition \ref{physwksoltns}. Suppose that $u_0 \in L^2$ is such that $\curl u_0 \equiv \omega_0 \in L^p(\torus^2)$, for some $p>1$.  Then, for all $0\leq t \leq T$,
\[\lim_{\nu \to 0} \|u^{\nu}(t,\cdot)\|_{L^2}^2 = \|u(t,\cdot)\|_{L^2}^2.\]
\end{lemma}

\begin{proof}
It follows easily from the vorticity formulation that $\curl u^{\nu} \equiv \omega^{\nu}$ belongs to a bounded subset of $L^{\infty}(0,T;L^p(\torus^2))$, so that, integrating, we find $u^{\nu} $ is bounded in $L^{\infty}(0,T;W^{1,p}(\torus^2))$. In addition, using the Navier-Stokes equations we can easily show that $u^{\nu}$ belongs to a bounded subset of $Lip(0,T;H^{-L}(\torus^2))$, for some (possibly large) $L>0$. Since $W^{1,p}(\torus^2)$ is compactly imbedded into $L^2(\torus^2)$, it is standard, hence, to deduce that $\{u^{\nu}\}$ lies in a compact subset of $C(0,T;L^2(\torus^2))$. Since
$u^{\nu} \rightharpoonup u$ weakly$^*$ in $L^{\infty}(0,T;L^2(\torus^2))$ it follows, by uniqueness of limits, that
\[u^{\nu} \to u, \;\;\;\;\;\; \mbox{ strongly in } C(0,T;L^2(\torus^2)).\]
This is enough to conclude the proof.

\end{proof}

We now proceed to the proof of Theorem \ref{physwksoltns}.

\begin{proof}
We assume without loss of generality that $\omega_0 \in L^p(\torus^2)$ for some $p<2$, and that $\omega_0 \notin L^2(\torus^2)$ as, otherwise, the result is trivial.
As $u$ is assumed to be a physically realizable weak solution, there exists a family  $\{u^{\nu}\}$ of solutions to the Navier-Stokes equations with viscosity $\nu$ satisfying the conditions of Definition \ref{physsol}.
Let $\omega^{\nu}=\curl u^{\nu}$. The vorticity equation reads:
\begin{equation} \label{vorteq}
\partial_t\omega^{\nu} + u^{\nu}\cdot\nabla\omega^{\nu} = \nu \Delta \omega^{\nu}.
\end{equation}
Multiplying the vorticity equation by $\omega^{\nu}$ and integrating on the torus yields
\begin{equation} \label{vortenest}
\frac{d}{dt}\|\omega^{\nu}\|_{L^2}^2 = -2\nu\|\nabla\omega^{\nu}\|_{L^2}^2.
\end{equation}
Use the Gagliardo-Nirenberg inequality to obtain, for any $1<p<2$,
\begin{equation} \label{GNforvort}
\|\omega^{\nu}\|_{L^2} \leq \|\nabla \omega^{\nu}\|_{L^2}^{1-\frac{p}{2}} \|\omega^{\nu}\|_{L^p}^{\frac{p}{2}}.
\end{equation}
It follows easily that
\begin{equation} \label{enstrineq}
-2\nu\|\nabla \omega^{\nu}\|_{L^2}^2 \leq -2\nu \|\omega^{\nu}\|_{L^2}^{\frac{4}{2-p}} \|\omega^{\nu}\|_{L^p}^{-\frac{2p}{2-p}}.
\end{equation}
We note, however, that if we multiply the vorticity equation by $|\omega^{\nu}|^{p-2}\omega^{\nu}$ and integrate on the torus, we obtain
a maximum principle for the $L^p$ norm of vorticity, namely:
\begin{equation} \label{vortmaxprincLp}
\|\omega^{\nu}(t,\cdot)\|_{L^p} \leq \|\omega_0^{\nu}\|_{L^p},
\end{equation}
for any $t\geq 0$.
Therefore, using \eqref{vortmaxprincLp} in \eqref{enstrineq}, we obtain from \eqref{vortenest}
\begin{equation} \label{vortenestfin}
\frac{d}{dt}\|\omega^{\nu}\|_{L^2}^2 \leq -2\nu \|\omega^{\nu}\|_{L^2}^{\frac{4}{2-p}} \|\omega_0^{\nu}\|_{L^p}^{-\frac{2p}{2-p}}.
\end{equation}
Write $y=y(t)=\|\omega^{\nu}\|_{L^2}^2$ and $C_0 = \|\omega_0^{\nu}\|_{L^p}^{-\frac{2p}{2-p}}$. Then, integrating \eqref{vortenestfin} in time, starting from $\delta > 0$, we obtain
\[
[y(t)]^{\frac{-p}{2-p}} - [y(\delta)]^{\frac{-p}{2-p}}\geq \frac{2\nu C_0 p}{2-p} (t - \delta).
\]
Taking the limit as $\delta \to 0$ and using that $\lim_{\delta \to 0} \|\omega^{\nu}(\delta,\cdot)\|_{L^2}^2 = +\infty$ we find that
\begin{equation} \label{vortL2finest}
\|\omega^{\nu}(t,\cdot)\|_{L^2}^2 \leq \left(\frac{2\nu p C_0 t}{2-p} \right)^{-\frac{2-p}{p}}.
\end{equation}
Next recall that solutions of the Navier-Stokes equations satisfy the energy identity in two space dimensions:
\begin{equation} \label{enestNS}
\frac{d}{dt}\|u^{\nu}\|_{L^2}^2 = -2\nu\|\nabla u^{\nu}\|_{L^2}^2.
\end{equation}
Rewriting the right-hand-side above in terms of vorticity yields
\begin{equation} \label{enestNSvort}
\frac{d}{dt}\|u^{\nu}\|_{L^2}^2 = -2\nu\|\omega^{\nu}\|_{L^2}^2.
\end{equation}
Hence, integrating \eqref{enestNSvort} in time and using \eqref{vortL2finest} we deduce that

\begin{eqnarray*}
0 \geq \|u^{\nu}(t,\cdot)\|_{L^2}^2 - \|u_0^{\nu}\|_{L^2}^2 & \geq & -2\nu \int_0^t \left(\frac{2\nu p C_0 s}{2-p} \right)^{-\frac{2-p}{p}}\,ds \\ \\
& = & -2\nu \left(\frac{2\nu p C_0 }{2-p} \right)^{-\frac{2-p}{p}} \frac{p}{2(p-1)} t^{\frac{2(p-1)}{p}},\\
\end{eqnarray*}
that is,
\begin{equation} \label{unuest}
0 \geq \|u^{\nu}(t,\cdot)\|_{L^2}^2 - \|u_0^{\nu}\|_{L^2}^2 \geq -(2\nu)^{\frac{2(p-1)}{p}} \left(\frac{p C_0 }{2-p} \right)^{-\frac{2-p}{p}} \frac{p}{2(p-1)} t^{\frac{2(p-1)}{p}}.
\end{equation}
Now, since $p>1$ the right-hand-side of \eqref{unuest} vanishes as $\nu \to 0$. Therefore,
\[
\lim_{\nu \to 0} \|u^{\nu}(t,\cdot)\|_{L^2}^2 - \|u_0^{\nu}\|_{L^2}^2 = 0.
\]
Using item (2b) of Definition \ref{physsol} together with Lemma \ref{noconcentrateviscous}, we complete the proof.

\end{proof}

We remark that \eqref{unuest}
provides a rate at which the loss of the energy of a viscous solution goes to zero, namely
\[
\|u_0^{\nu}\|_{L^2}^2  - \|u^{\nu}(t,\cdot)\|_{L^2}^2\leq  C_p \nu^{\frac{2(p-1)}{p}}  t^{\frac{2(p-1)}{p}} \|\omega_0^{\nu}\|_{L^p}^2,
\]
where $C_p$ is an absolute constant that depends only on $p$.

Finally, we note that the same condition, namely that $u\in C(0,T;L^2(\torus^2))$ be a physically realizable weak solution such that, initially, its vorticity $\curl u_0 \equiv \omega_0 \in L^p(\torus^2)$, for some $p>1$, was needed to prove that the vorticity be a {\it renormalized} solution of the vorticity equation and, in particular, that it preserves its $L^p$-norm, see \cite{CS14}.

  \end{document}